\documentclass[11pt]{article}
\usepackage[utf8]{inputenc}
\usepackage[margin=1.25in]{geometry}
\usepackage[fleqn]{amsmath}
\usepackage{amssymb}
\usepackage{graphicx}       
\usepackage{amsthm}  
\usepackage{bbm}
\usepackage{xcolor}
\usepackage{enumerate}
\numberwithin{equation}{section} 
\usepackage[fleqn]{cases}
\usepackage{cases}
\newtheorem{remark}{Remark}
\newtheorem{lemma}{Lemma}[section]
\newtheorem{theorem}{Theorem}[section]
\newtheorem{definition}{Definition}[section]

\newtheorem{proposition}{Proposition}[section]

\newtheorem{corollary}{Corollary}[section]

\def\e{{\epsilon}}
\def\F{{\mathcal F}}

\def\P{{\mathbb P}}
\def\E{{\mathbb E}}
\def\R{{ \mathbb R}}

\def\vs{\vspace{1mm}}

\def\ds{\displaystyle}
\newcommand{\norm}[1]{\left\lVert#1\right\rVert}
\def\Z20{{\mathbb{Z}^2_0}}
\def\T2{\mathbb{T}^2}
\def\hs{\hspace{2mm}}
\def\vs{\vspace{2mm}}

\begin{document}
\title{Large deviations principle for the invariant measures of the 2D stochastic Navier-Stokes equations with vanishing noise correlation}
\author{S. Cerrai, N. Paskal}
\date{}

\maketitle
\begin{abstract}
We study the two-dimensional incompressible Navier-Stokes equation on the torus, driven by Gaussian noise that is white in time and colored in space. We consider the case where the magnitude of the random forcing $\sqrt{\e}$ and its correlation scale $\delta(\e)$ are both small. We prove a large deviations principle for the solutions, as well as for the family of invariant measures, as $\e$ and $\delta(\e)$ are simultaneously sent to $0$, under a suitable scaling. 
\end{abstract}

\section{Introduction}
In the present paper, we consider the  two-dimensional incompressible Navier-Stokes equation on the torus $\T2 =[0,2\pi]^2$, perturbed by a small additive noise
\begin{equation}\label{eq_NS_bad}
	\begin{cases}
		\partial_t u(t,x)+ (u (t,x)\cdot \nabla)u(t,x)= \Delta u(t,x) + \nabla p(t,x)+  \sqrt{\e \,Q_\e}\ \partial_t \xi(t,x), 
		\\[10 pt] \mathrm{div} \ u(t,x) = 0,   \hs \hs u(0,x) = u_0(x), \hs \hs u \text{ is periodic}.
	\end{cases}
\end{equation}
The functions $u(t,x) \in \R^2$ and $p(t,x) \in \R$ denote, respectively,  the velocity and the pressure of the fluid at any $ (t,x) \in \R^+ \times \T2$. The random forcing $\partial_t \xi(t,x)$ is a space-time white noise, while the operator $\sqrt{Q_\e}$ provides spatial correlation to the noise on a scale of size $\delta(\e)$. Here, we are interested in the behavior of equation $\eqref{eq_NS_bad}$ as the noise magnitude $\sqrt{\e}$ and the correlation scale $\delta(\e)$ are simultaneously sent to $0$. 

In two dimensions,  the incompressible Navier-Stokes equation driven by space-time white noise is well-posed only in spaces of negative regularity (see \cite{dd}). The driving noise must have more regularity in the spatial variable in order to have function-valued solutions. In our case, we consider a smoothing operator $\sqrt{Q_\e}$ that provides sufficient regularity to interpret equation $\eqref{eq_NS_bad}$ in the space $C([0,T];[L^2(\T2)]^2)$, for any fixed $\e > 0$. In fact, the regularization $\sqrt{Q_\e}$ can be chosen to decay to the identity operator slowly enough for the $\sqrt{\e}$ factor to compensate and produce a function-valued limit. 

Under the present assumptions, the $\e \downarrow 0$ limit of equation $\eqref{eq_NS_bad}$ in $C([0,T];[L^2(\T2)]^2)$ is unsurprisingly the corresponding unforced Navier-Stokes equation. A more interesting problem is the quantification of the convergence rate, which can be done using large deviations theory. In \cite{cd}, it was shown that the solutions to the Leray-projected version of equation $\eqref{eq_NS_bad}$ satisfy a large deviations principle in $C([0,T];[L^2(\T2)]^2)$ with rate function
\begin{equation*}
	I(u) = \frac{1}{2} \int_0^T \norm{u'(t)+Au(t)+B(u(t))}_{[L^2(\T2)]^2}^2 dt,
\end{equation*}
where $A$ is the Stokes operator and $B$ is the Navier-Stokes nonlinearity. This result was proven using the weak convergence approach developed in \cite{bdm}, which is particularly effective at handling multiple parameter limits. The weak convergence method was also used in \cite{bm1} and \cite{bm2} to prove large deviations principles for the stochastic Navier-Stokes with viscosity vanishing at a rate proportional to the strength of the noise, which is believed to be a relevant problem in the study of turbulent fluid dynamics. 

If the operator $\sqrt{Q_\e}$ is simultaneously smoothing enough but not too degenerate, then equation $\eqref{eq_NS_bad}$ will possess a unique ergodic invariant probability measure (see \cite{f}). In the $\e \downarrow 0$ limit, it can be shown that these measures converge weakly to the Dirac measure at $0$. For fixed correlation strength $\delta(\e) = \delta > 0$, it was proven in \cite{bc} that the invariant measures also satisfy a large deviations principle in $[L^2(\T2)]^2$ with rate function given by the quasi-potential
\begin{equation*}
	U_\delta(x) = \inf \left\{I^\delta_T(u)\,:\,T > 0,\ u \in C([0,T];[L^2(\T2)]^2),\ u(0) = 0,\ u(T) =x \right\},
\end{equation*}
where $I^\delta_T:C([0,T];[L^2(\T2)]^2) \to [0,+\infty]$ is the action functional for the paths, defined by
\begin{equation*}
	I^\delta_T(u) := 
		\frac{1}{2} \int_0^T \norm{Q_\delta^{-1} \Big(u'(t) + Au(t) + B(u(t)) \Big)}_{[L^2(\T2)]^2}^2dt.
\end{equation*}
This result was generalized in \cite{m} to the case of the Navier-Stokes equations posed on a bounded domain with Dirichlet boundary conditions. In \cite{m} they also considered the case where the equation has a deterministic, time-independent forcing so that the limiting dynamics may be nontrivial point attractors or sets of attractors. Both papers established their results by following the general strategy introduced in \cite{s2} for proving large deviations principles for families of invariant measures.
 
In \cite{bcf}, it was also proven that the quasipotential $U_\delta(x)$, corresponding to the problem on the torus, converges pointwise to 
\begin{equation*}
U(x) = \norm{x}_{[H^1(\T2)]^2}^2,
\end{equation*}
 as $\delta \downarrow 0$. This is a consequence of the orthogonality of $Au$ and $B(u)$ in $[L^2(\T2)]^2$, which in general does not hold for the problem posed on a bounded domain. In some sense, $U(x)$ is what one would expect the quasi-potential for the  space-time white noise case to be, if the time-stationary problem were well-posed.

The purpose of this article is to bridge the results of \cite{bc} and \cite{bcf} with the result of \cite{cd}. Rather than first taking $\e \downarrow 0$ and then studying what happens as the regularization is removed, we take $\e$ and $\delta$ to $0$ simultaneously. We prove that the invariant measures of equation $\eqref{eq_NS_bad}$ satisfy a large deviations principle directly with rate function $U(x)$, under suitable conditions on the regularization $\sqrt{Q_\e}$.

To prove this result, we first prove a large deviations principle for the solutions of equation $\eqref{eq_NS_bad}$ in $ C([0,T] ;[L^2(\T2)]^2)$, that is uniform with respect to initial conditions in appropriate sets of functions. This is done by proving a large deviations principle for the linearized problem using the weak convergence approach and then transferring this to the nonlinear problem via the contraction principle. We note that this different method allows for slower decay of the correlation scale $\delta(\e)$ than the one introduced in   \cite{cd}. The proof of the large deviations principle for the invariant measures then follows for some points along the same lines as \cite{bc}, but requires several crucial modifications to account for the decaying regularity of the driving noise.

\section{Preliminaries}
We consider equation $\eqref{eq_NS_bad}$ posed on the space of square-integrable, mean zero, space-periodic functions. For an introduction to the 2D Navier-Stokes equations on the torus, see the book \cite{t} by Temam. We follow the notations and conventions used there. Denoting $\mathbb{T}^2:= [0,2\pi]^2$, we define
\begin{equation*}
	H:= \Big\{ f \in [L^2(\T2)]^2: \int_{\T2} f(x)dx = 0, \hs \mathrm{div}  f = 0, \hs f  \text{ is periodic in } \T2\Big\},
\end{equation*}
where the periodic boundary conditions are interpreted in the sense of trace. It can be shown that $H$ is a Hilbert space when endowed with the standard $L^2(\T2)$ inner product. We denote the norm and inner product on $H$ by $\norm{\cdot}_H$ and $\langle \cdot,\cdot \rangle_H$, respectively. Moreover, in what follows, for every $p\geq 1$ we shall write $L^p$ instead of $L^p(\T2)$.

We denote by $H_\mathbb{C}$, the complexification of $H$, and by  $\Z20$ the set $\mathbb{Z}^2 \setminus \{(0,0)\}$. The family $\{e_k\}_{k \in \Z20} \subset H_\mathbb{C}$ defined by
\begin{equation*}
	e_k(x) = \frac{1}{2\pi} \frac{(k_2,-k_1)}{\sqrt{k_1^2+k_2^2}}\ e^{i x \cdot k},\ \ \ \ \  x \in \mathbb{T}^2,  \ \  k= (k_1,k_2)\in \Z20,
\end{equation*}
form a complete orthonormal system in $H_\mathbb{C}$. Similarly, the family $\{\mathrm{Re}(e_k)\}_{k \in \Z20} \subset H$ form a complete orthonormal system in $H$. In what follows, we use the basis $\{e_k\}_{k \in\,\Z20}$ with the implicit assumption that we are only considering the real components. 

Next, we let $P$ be the orthogonal projection from $[L^2(\T2)]^2$ onto $H$, known as the Leray projection. We define the Stokes operator by setting 
\begin{equation*}
	Au := -P \Delta u,\ \ \ \  u \in\,D(A):= H \cap [W^{2,2}(\T2)]^2.
\end{equation*}
It is easy to see that $A$ is a diagonal operator on $H$ with respect to the basis $\{e_k\}_{k \in \Z20}$. In particular, for any $k \in \Z20$ we have
\begin{equation*}
	Ae_k = |k|^2e_k.
\end{equation*} 
Since $A$ is a positive, self-adjoint operator, for any $r \in \R$ we can define the fractional power $A^r$ with domain $D(A^r)$. In fact, it can be shown that $D(A^r)$ is the closure of $\mathrm{span}_{\,k \in \Z20}\,\langle e_k\rangle$ with respect to the $[W^{2r,2}(\T2)]^2$ Sobolev norm. To simplify our notations, we will denote $V^r:= D(A^{r/2})$, with the norm given by the $[W^{2r,2}(\T2)]^2$ Sobolev semi-norm
\begin{equation*}
	\norm{u}_r^2 := \norm{u}_{D(A^{r/2})}^2 = \norm{u}_{[H^r(\T2)]^2 }^2 = \sum_{k \in \Z20} |k|^{2r} \langle u, e_k \rangle_H^2.
\end{equation*}
In particular, we have that $V^2 = D(A)$ and $V:= V^1 = D(A^{1/2})$. For any $r \geq 0$, we denote by $V^{-r}$  the dual space of $V^r$. In addition, for any $p \geq 1$, we will use the shorthands
\begin{equation*}
 	L^p:= [L^p(\T2)]^2, \ \ \ \  W^{k,p}:= [W^{k,p}(\T2)]^2.
\end{equation*}

Next, we define the tri-linear form, $b:V \times V \times V \to \R$, by
\begin{equation*}
	b(u,v,w) := \int_{\T2} (u(x)\cdot \nabla)v(x) \cdot w(x)dx, \hs u,v,w \in V.
\end{equation*}
From standard interpolation inequalities and Sobolev embeddings, it follows that
\begin{equation}
\label{eq_sob1}	
|b(u,v,w)| \leq c\begin{cases}
		&\norm{u}_H^{1/2}\norm{u}_V^{1/2}\norm{v}_V\norm{w}_H^{1/2}\norm{w}_V^{1/2},\\[10pt]
		&\norm{u}_H^{1/2}\norm{u}_{V^2}^{1/2}\norm{v}_V \norm{w}_H,\\[10pt]
		&\norm{u}_H \norm{v}_V \norm{w}_H^{1/2} \norm{w}_{V^2}^{1/2},\\[10pt]
		&\norm{u}_H^{1/2} \norm{u}_V^{1/2} \norm{v}_V \norm{w}_H^{1/2} \norm{w}_V^{1/2},
		\end{cases}
\end{equation}
for smooth $u,v,w$. These inequalities can then be extended to the appropriate Sobolev spaces by continuity. We note that the first inequality in  $\eqref{eq_sob1}$ implies that $b$ is indeed well-defined and continuous on $V \times V \times V$. The tri-linear form $b$ also induces the continuous mappings $B: V \times V \to V'$ and $B:V \to V'$ defined by
\begin{align*}
	 & \langle B(u,v),w \rangle := b(u,v,w),
	\\[10pt]
	  &B(u) := B(u,u),
\end{align*}
for $u,v,w \in V$. It can  be shown that for any $u,v \in D(A)$
\begin{equation*}
	B(u,v)= P [(u\cdot \nabla)v],
\end{equation*} 
and
\begin{equation}
\label{eq_prop_non_negs}
\langle B(u,v),w \rangle_H = - \langle B(u,w),v \rangle_H,  \ \ \ \ u,v,w \in V.\end{equation}
Moreover
\begin{equation} \label{eq_prop_non_Au} \langle B(u),Au \rangle_H = 0, \ \ \ \  u \in D(A),
\end{equation} which implies
that
\[\langle B(u,v),v \rangle_H =0, \ \ \ \ u, v \in\,V.\]

Equation $\eqref{eq_prop_non_negs}$ is still true when considering the problem posed on a bounded domain with Dirichlet boundary conditions. Equation $\eqref{eq_prop_non_Au}$, on the other hand, only holds for the problem posed on the torus with periodic boundary conditions (for a proof of  $\eqref{eq_prop_non_Au}$, see for example \cite{ks}). We note that the proof of our main result relies on equation $\eqref{eq_prop_non_Au}$ in several places, and hence will not immediately generalize to the case of the Navier-Stokes equation on a bounded domain. 

As for the random forcing in equation $\eqref{eq_NS_bad}$, we assume that $\xi(t,x)$ is a cylindrical Wiener process on the Hilbert space of mean-zero functions in $[L^2(\T2)]^2$. We then set $w(t):= P\xi(t)$, so that $w$ has the formal expansion
\begin{equation*}
	w(t,x) = \sum_{k \in \Z20} e_k(x) \beta_k(t), \ \ \ \  t \geq 0, \ \ \ x \in \T2,
\end{equation*}
where  $\{\beta_k\}_{k \in \Z20}$ are a collection of independent, real-valued Brownian motions on some filtered probability space $(\Omega,\mathcal{F},\{\mathcal{F}_t\}_{t \geq 0}, \P)$. We assume that the covariance operator $Q_\e$ belongs to $\mathcal{L}(H;H)$ and takes the  form
\begin{equation}\label{eq_Q_cov}
	Q_\e := (I + \delta(\e) A^\beta)^{-1},
\end{equation}
for some $\beta > 0$ and $\delta(\e) > 0$. Since we are concerned with the singular noise limit, $\delta(\e)$ will be taken to be a strictly decreasing function of $\e$ such  that 
\[\lim_{\e \to 0} \delta(\e) = 0.\]
 Definition $\eqref{eq_Q_cov}$ implies that $Q_\e$ is diagonal with respect to the basis $\{e_k\}_{k \in \Z20}$. 
 \begin{remark}
 {\em In the present paper  we only take this particular form of the covariance operator in order to simplify our presentation. The results below can easily be adapted to more general covariance operators with the same smoothing and ergodic properties.}
	
 \end{remark}

The driving noise, $\sqrt{Q_\e}\,w(t)$, can thus formally be written as the infinite series
\begin{equation*}
	\sqrt{Q_\e}\,w(t,x) =  \sum_{k \in \Z20} \sigma_{\e,k}\,e_k(x) \,\beta_k(t) := \sum_{k \in \Z20} (1+\delta(\e) |k|^{2\beta})^{-1/2} e_k(x) \,\beta_k(t).
\end{equation*}
Since $\delta(\e)$ converges to zero, as $\e\downarrow 0$, the covariance operator $Q_\e$ convergences pointwise to the identity operator, as $\e \downarrow 0$. For each fixed $\e > 0$, it is immediate to check that $\sqrt{Q_\e}  \in\,\mathcal{L}(V^r,V^{r+\beta})$. In fact, one can show that 
\begin{equation}\label{eq_Q_reg}
	\norm{\sqrt{Q_\e}f }_{V^{r+q}} \leq \frac{1}{\sqrt{\delta(\e)}}\, \norm{f}_{V^r},
\end{equation}
for any $r \in \R$, $q \leq \beta$ and $f \in V^r$. Moreover, $Q_\e$ is a trace class operator in $H$ if and only if $\beta > 1$. This means that the Wiener process $\sqrt{Q_\e}\,w$ is $H$-valued only when $\beta > 1$. 
 
By taking the Leray projection on both sides of equation $\eqref{eq_NS_bad}$, we obtain the following stochastic evolution problem
\begin{equation}\label{eq_NS_good}
	\begin{cases}
		du(t) + \left[Au(t) +B(u(t))\right]dt = \sqrt{\e\,Q_\e}\, dw(t),
		\\[10pt] u(0) = x.
	\end{cases}
\end{equation}
We assume the initial condition $x$ is an element of $H$. As is well-known (see \cite{bl} or Chapter 15 of \cite{dz1}), under the assumption that $\beta > 0$, equation $\eqref{eq_NS_good}$ admits a unique generalized solution, $u_\e^x \in C([0,T];H)$. That is, there exists a progressively measurable process $u_\e^x$ taking values in $ C([0,T];H)$, $\P$-a.s. for any $T > 0$, such that
\begin{align*}
	\langle u^x_\e(t),h\rangle_H & = \langle x,h \rangle_H - \int_0^t \langle u^x_\e(s), Ah \rangle_H
	\\ & - \int_0^t \langle B(u^x_\e(s),h),u_\e(s)\rangle_H + \langle \sqrt{\e\,Q_\e}\,w(t),h\rangle_H, \quad \P-a.s.,
\end{align*}
for any $h \in D(A)$ and $t \in [0,T]$.

The condition $\beta > 0$ is not enough to ensure the existence and uniqueness of an invariant measure for equation $\eqref{eq_NS_good}$.  In the last twenty five years there has been an extremely intense activity aimed to the study of the ergodic properties of randomly perturbed PDEs in fluid dynamics and, in particular, of equation \eqref{eq_NS_bad}.  As shown  for instance in the monograph $\cite{ks}$, a sufficient condition for this is that $Q_\e$ be trace-class in $H$ and $\sigma_{\delta(\e),k} \neq 0$ for all $k$. Notice that if $\beta>1$, then $Q_\e$ is a trace-class operator and by applying It\^o's formula  we get
\begin{equation}
\label{ns1}
\mathbb{E}\,\Vert u^x_\e(t)\Vert_H^2+2\int_0^t\mathbb{E} \Vert u^x_\e(s)\Vert_V^2\,ds=\Vert x\Vert_H^2+t\,\e\, \mbox{Tr}\,Q_\e\leq \Vert x\Vert_H^2+t\,\e\,\delta_\e^{-1/\beta}.
\end{equation}
This means that $u^x_\e \in\,L^2(\Omega;C([0,T];H)\cap L^2(0,T;V))$ and, in particular, for every $\e>0$ there exists an invariant measure.

Now, let $\{\nu_\e\}_{\e > 0}$ be this family of invariant measures. Each $\nu_\e$ is ergodic in the sense that
\begin{equation*}
\lim_{T \to \infty} \frac{1}{T} \int_0^T f(u_\e^x(t))dt = \int_H f(h)d\nu_\e(x),
\end{equation*}
for all $x \in H$ and Borel-measurable $f:H \to \R$. If
\begin{equation}
	\label{ns2}
	\sup_{\e \in\,(0,1)} \e\,\delta_\e^{-1/\beta}<\infty,
\end{equation}
we have that the family $\{\nu_\e\}_{\e>0}$ is tight in $H$. Actually, due to \eqref{ns1} and the invariance of $\nu_\e$, for every $T>0$ we have
\[\begin{aligned}
\int_H \Vert x\Vert_{V}^2\,d\nu_\e(x) & =\frac 1T \int_0^T \int_H\mathbb{E}\,\Vert u^x_\e(t)\Vert_{V}^2\,d\nu_\e(x)\,dt=\frac 1T\int_H \int_0^T \mathbb{E}\,\Vert u^x_\e(t)\Vert_{V}^2\,dt\,d\nu_\e(x)\\[10pt]
& \leq \frac 1{2\,T}\int_H\Vert x\Vert_H^2\,d\nu_\e(x)+\frac 12\,\e\,\delta_\e^{-1/\beta}\leq  \frac 1{2\,T}\int_H\Vert x\Vert_V^2\,d\nu_\e(x)+\frac 12\,\e\,\delta_\e^{-1/\beta}.
\end{aligned}\]
Then, thanks to \eqref{ns2}, if we choose $T>1$ we get
\[\sup_{\e \in\,(0,1)}\int_H \Vert x\Vert_{V}^2\,d\nu_\e(x)<\infty,\] and this implies the tightness of $\{\nu_\e\}_{\e \in\,(0,1)}$ in $H$. In fact, provided that
\[\lim_{\e\to 0} \e\,\delta_\e^{-1/\beta}=0,\]
we have that
\[\nu_\e\rightharpoonup \delta_0,\ \ \ \ \text{as}\ \e\downarrow 0.\]

 The purpose of this paper is to quantify the rate of this convergence through a large deviations principle. To state the main result, we first recall the definition of the large deviations principle. Here we give the Freidlin-Wentcell formulation.

\begin{definition}
	Let $E$ be a Banach space. Suppose that $\{\mu_\e\}_{\e > 0}$ is a family of proability measures on $E$ and $I:E \to [0,+\infty]$ is a good rate function, meaning that for each $ s \geq 0$, the level set $\Phi(s):=\{h \in E:I(h)\leq s\}$ is a compact subset of $E$. The family $\{\mu_\e\}_{\e > 0}$ is said to satisfy a large deviations principle (LDP) in $E$, with rate function $I$, if the following hold.
	\begin{enumerate}[(i)]
		\item For every $s \geq 0$, $\delta > 0$ and $\gamma > 0$, there exists $\e_0 > 0$ such that
		\begin{equation*}
			\mu_\e(B_E(\varphi,\delta)) \geq \exp \left(-\frac{I(\varphi) + \gamma}{\e} \right),
		\end{equation*}
		for any $\e \leq \e_0$ and $\varphi \in \Phi(s)$, where $B_E(\varphi,\delta):= \{h\in E:\norm{h-\varphi} < \delta\}$.
		\item For every $s_0 \geq 0$, $\delta > 0$ and $\gamma > 0$, there exists $\e_0 > 0$ such that
		\begin{equation*}
			\mu_\e( B^c_E(\Phi(s),\delta)) \leq \exp \left( -\frac{s-\gamma}{\e}\right),
		\end{equation*}
		for any $\e \leq \e_0$ and $s \leq s_0$, where $B^c_E(\Phi(s),\delta) := \{h \in E: \mathrm{dist}_E(h,\Phi(s)) \geq \delta\}$.
\end{enumerate}	 
\end{definition}
\noindent The main result of this paper is the following.
\begin{theorem}\label{th_main}
	Assume that $Q_\e$ has the form given in \eqref{eq_Q_cov}, for some $\beta > 2$. Moreover, suppose that 
	\begin{equation*}
		\lim_{\e \to 0} \delta(\e) = 0,\ \ \ \ \lim_{\e \to 0} \e\, \delta(\e)^{-2/\beta} = 0.
	\end{equation*}
	Then the family of invariant measures $\{\nu_\e\}_{\e > 0}$ of equation $\eqref{eq_NS_good}$ satisfies a large deviations principle in $H$ with rate function given by
	\begin{equation}
	\label{fg}	
	U(x) =
			\begin{cases}
				\norm{x}_V^2, &  x \in V,
				\\[10pt] +\infty, &  x \in V \setminus H.
			\end{cases}	
	\end{equation}

\end{theorem}

We remark here that the rate function, $U(x)$, is really the quasipotential corresponding to equation $\eqref{eq_NS_good}$, whose definition is given in equation $\eqref{eq_quasi}$. The quasi-potential has the explicit representation given in  \eqref{fg} in the case  the problem is posed on a torus. That formula does not hold in general for the problem posed on a bounded domain with Dirichlet boundary conditions.

\section{Large deviation principle for the paths}
The proof of Theorem \ref{th_main} requires a large deviations principle for the solutions to equation $\eqref{eq_NS_good}$. One such large deviations principle  is proven in \cite{cd}, but here we have to  proceed differently  in order to obtain a result that is uniform with respect to initial conditions in bounded subsets of $H$. Unlike in \cite{cd}, we first prove a large deviation principle for the linearized Ornstein-Uhlenbeck process in the space $C([0,T];L^4)$, and then transfer it back to the appropriate Navier-Stokes process by means of the contraction principle. 

\subsection{LDP for the Ornstein-Uhlenbeck process} Assume that $Q_\e$ has the form given in $\eqref{eq_Q_cov}$, for some $\beta > 0$. For every $\e>0$, let $z_\e$ denote the  mild solution to the equation
\begin{equation}\label{eq_OU_process}
	\begin{cases}
			dz_\e + Az_\e dt = \sqrt{\e\,Q_\e}\, dw(t),
			\\[10pt] z_\e(0) = 0.
	\end{cases}
\end{equation}
It is well-known that  $z_\e$ is given by the stochastic convolution
\begin{equation*}
	z_\e(t) = \int_0^t S(t-s)\sqrt{\e\,Q_\e}\, dw(s), \ \ \ \  t \geq 0,
\end{equation*}
where $\{S(t)\}_{t \geq 0}$ is the analytic semigroup generated by the operator $-A$ on $H$. It can be shown that $z_\e \in L^p(\Omega;C([0,T];V^r))$, for any $r < \beta$ and $p\geq 1$ (e.g. see \cite{dz2}). In this subsection, we  prove that the family $\{z_\e\}_{\e > 0}$ satisfies a large deviations principle in $C([0,T];L^4)$. To do so, we first prove that the stochastic convolution $z_\e$ converges to $0$ in $L^p*(\Omega;C([0,T];L^p))$, as $\epsilon \downarrow 0$, for every $p\geq 1$.
\begin{lemma}\label{lem_convolution_to0}
	For any $\e > 0$, the process $z_\e$ has trajectories in $C([0,T];L^p)$, $\P$-a.s. for any $p \in [1,\infty)$. Moreover, 
	\begin{equation}
	\label{eq_Lp_limit}
		\lim_{\e \to 0} \e \log \frac 1{\delta(\e)}= 0\ \Longrightarrow\  \lim_{\e \to 0} \E \sup_{t  \in\,[0,T]} \norm{z_\e(t)}_{L^p}^p = 0.
	\end{equation}	
\end{lemma}
\begin{proof}
	Fix any $p < \infty$. Thanks to the Burkholder-Davis-Gundy inequality and the uniform boundedness of the basis $\{e_k\}_{k \in \Z20}$, we have
	\begin{align*}
		\E \sup_{0 \leq t \leq T} \norm{z_\e(t)}_{L^p}^p & =   \e^{p/2}\, \E \sup_{t \in\,[0,T]} \norm{\int_0^t S(t-s)\sqrt{Q_\e}dw(s)}_{L^p}^p
		\\ & \leq \e^{p/2}\, \int_{\T2} \E \sup_{t \in\,[0,T]} \Big| \int_0^t\sum_{k \in \Z20} e^{-|k|^2(t-s)}\sigma_{\delta(\e),k} e_k(x)d\beta_k(s)\Big|^p dx 
		\\ & \leq c_p \ \e^{p/2}\, \int_{\T2} \Big(\sum_{k \in \Z20} \sigma_{\delta(\e),k}^2|e_k(x)|^2\int_0^T e^{-2|k|^2s}ds \Big)^{p/2}dx 
		\\ & \leq c_p\, \e^{p/2}\, \Big( \sum_{k \in \Z20} \frac{1}{|k|^2(1+\delta(\e) |k|^{2\beta})} \Big)^{p/2}<\infty.
	\end{align*}
	 Therefore,  $\eqref{eq_Lp_limit}$ follows by noting that
	\begin{align*}
		\sum_{k \in \Z20} \frac{1}{|k|^2(1+\delta(\e) |k|^{2\beta})} & \leq \int_1^\infty \frac{1}{r(1+\delta(\e) r^{\beta})}dr = \int_{\delta(\e)^{1/(\beta)}}^\infty \frac{1}{r(1+r^{\beta})}dr  
		\\ & \leq \int_{\delta(\e)^{1/(\beta)}}^1 \frac{dr}{r}  + \int_1^\infty \frac{dr}{r^{\beta+1}} \leq \frac{1}{\beta} \log \frac{1}{\delta(\e)} + \frac{1}{\beta}.
	\end{align*}	
\end{proof}
To prove that the family $\{z_\e\}_{\e > 0}$ satisfies a large deviations principle in $C([0,T];L^4)$ we use the weak convergence approach, as developed for SPDEs in \cite{bdm}. This approach involves proving convergence of the solutions to a sequence of controlled versions of the equations. For $\varphi \in L^2(\Omega;L^2(0,T;H))$, we denote by $z_{\e,\varphi}$ the solution to the equation
	\begin{equation*}
		dz_{\e,\varphi}(t) + Az_{\e,\varphi}(t)\,dt = \sqrt{\e\,Q_\e}\,dw(t) + \sqrt{Q_\e}\, \varphi(t)\,dt,\ \ \ \ z_{\e,\varphi}(0) = 0,	\end{equation*}
	and we denote by $z_\varphi$ the solution to the so-called skeleton equation
	\begin{equation}\label{eq_NS_forced_determ}
		\frac{d z_{\varphi}}{dt}(t) + Az_{\varphi}(t) =  \varphi(t),\ \ \ \ \ 
			z_\varphi(0) = 0.
	\end{equation}
Theorem 6 of \cite{bdm} implies the following result.
\begin{theorem}\label{th_laplace}
The family $\{\mathcal{L}(z_\e) \}_{\e > 0}$ satisfies a large deviations principle in $C([0,T];L^4)$, with  rate function
	\begin{equation}\label{eq_actionOU_bad}
		J_T(z) = \frac{1}{2}  \inf \Big\{ \int_0^T \norm{\varphi(t)}_H^2dt\ :\ \varphi \in L^2(0,T;H), z = z_\varphi \Big\},
	\end{equation}
	if the following two conditions hold for any $M \in [0,\infty)$.
	\begin{enumerate}[(i)]
		\item The set
		\begin{equation*}
			\Phi(M):= \Big\{\,z \in C([0,T];L^4)\, :\, z = z_\varphi,\hs \varphi \in L^2(0,T;H), \ \frac{1}{2}\int_0^T \norm{\varphi(t)}_H^2dt \leq M\,\Big\}
		\end{equation*}
		is a compact subset of $C([0,T];L^4)$.
		\item For every  $\{\varphi_\e\}_{\e \geq 0}\subset   L^2(\Omega;L^2(0,T;H))$, such that 
		\begin{equation}
				\label{sn11}
            \sup_{\e \in\,(0,1)}\frac{1}{2} \int_0^T \norm{\varphi_\e(t)}_H^2dt \leq M, \ \P-\text{a.s.},
		\end{equation}
		if $\varphi_\e$ converges to  $\varphi_0$ in distribution with respect to the weak topology of $L^2(0,T;H)$, as $\e \downarrow 0$, then $z_{\e,\varphi_\e}$ converges to $z_{\varphi_0}$ in distribution in $C([0,T];L^4)$, as $\e \downarrow 0$.
	\end{enumerate}
	
\end{theorem}
\noindent Thus, to prove the large deviations principle it remains to prove conditions (i) and (ii) in the above theorem. Note that condition (i) is precisely the statement that $J_T$ is a good rate function.
\begin{theorem}\label{th_OULDP}
	Assume that  
	\begin{equation*}
		\lim_{\e \to 0} \delta(\e) = 0,\ \ \ \ \lim_{\e \to 0} \e \log \frac{1}{\delta(\e)} = 0.
	\end{equation*} 
	Then the family $\{\mathcal{L}(z_\e) \}_{\e > 0}$ of solutions to equation $\eqref{eq_OU_process}$ satisfies a large deviations principle in $C([0,T];L^4)$ with rate function
	\begin{equation}\label{eq_actionOU_good}
		J_T(z) = 
		\begin{cases}
			\frac{1}{2} \ds \int_0^T\norm{z'(t)+Az(t)}_H^2dt & \text{ if } z \in W^{1,2}(0,T;H) \cap L^2(0,T;D(A)),
			\\ +\infty & \text{ otherwise}.
		\end{cases}		
	\end{equation}
\end{theorem}
\begin{proof}		
	In view of Theorem $\ref{th_laplace}$, it suffices to show that conditions (i) and (ii) in Theorem \ref{th_laplace} hold true. Equality of the rate functions defined in equations $\eqref{eq_actionOU_bad}$ and $\eqref{eq_actionOU_good}$ follows immediately from the fact that $z_\psi = z_\varphi$ implies that $\psi = \varphi$.
	
	\vs
	\noindent {\em Step 1.} We first verify condition (i). Suppose that $z \in \Phi(M)$, so that $z = z_\varphi$ for some $\varphi \in L^2(0,T;H)$ satisfying 
	\begin{equation}
	\label{sn10}
		\frac{1}{2}\int_0^T \norm{\varphi(t)}_H^2dt \leq M.
	\end{equation}
	For any $\zeta \in (0,1)$, the function $z_\varphi(t) = \int_0^t S(t-s) \varphi(s)ds$ can be rewritten as $z_\varphi = \Gamma_\zeta(Y_\zeta(\varphi))$ where
	\begin{equation*}
		\Gamma_\zeta(Y)(t):= c_\zeta \int_0^t (t-s)^{\zeta-1}S(t-s) Y(s)ds,
	\end{equation*}	
	and
	\begin{equation*}
		Y_\zeta(\varphi)(s):= \int_0^s (s-r)^{-\zeta}S(s-r)\varphi(r)dr.
	\end{equation*}
	It is possible to show that  for any $\zeta \in (0,\frac{1}{2})$, $p \geq 2 $, $\rho \in (0,1)$ and $\delta \in (0,\frac{1}{2})$ such that $ \ds \delta + \frac{\rho}{2} < \zeta - \frac{1}{p}$, 
	\begin{equation}\label{eq_Gamma_cont}
		\Gamma_\zeta:L^p(0,T;H) \to C^\delta([0,T];V^\rho),
	\end{equation}
	is a continuous linear mapping (see Appendix A of \cite{dz1}).	Moreover, by Young's inequality we have  that
	\begin{align}
		\nonumber \norm{Y_\zeta(\varphi)}_{L^p(0,T;H)}^p & = \int_0^T \norm{\int_0^s(s-r)^{-\zeta}S(s-r)\varphi(r)dr }_H^pdt
		\\[10pt] & \nonumber \leq \int_0^T \Big(\int_0^s(s-r)^{-\zeta} \norm{\varphi(r)}_Hdr \Big)^pdt 
		\\[10pt] & \label{eq_gamma_Lp} \leq \Big(\int_0^T t^{-\frac{2\zeta p }{p+2}}dt \Big)^\frac{p+2}{2}\norm{\varphi}_{L^2(0,T;H)}^p,
	\end{align}
	which is finite provided that $\zeta <\frac{1}{2} + \frac{1}{p}$. Hence $z \in C^{\delta}(0,T;V^{\rho})$ for any $\delta,\rho$ satisfying $\delta + \frac{\rho}{2} < \frac{1}{2}$. Moreover, thanks to \eqref{sn10}, 
	\begin{equation*}
		\norm{z}_{C^{\delta}(0,T;V^{\rho})} \leq c_{p,\delta,\rho} \sqrt{M},
	\end{equation*}
	so that $\Phi(M)$ is a bounded subset of $C^{\delta}(0,T;V^\rho)$ and thus a compact subset of $C([0,T];L^4)$.
	
		{\em Step 2.} Next, we verify condition (ii) in Theorem \ref{th_laplace}. Let $M > 0$ and let $\{\varphi_\e\}_{\e \geq 0}$ be a sequence in $L^2(\Omega;L^2(0,T;H))$ satisfying \eqref{sn11}. Thanks to the Skorohod theorem, there exists a probability space $(\bar{\Omega},\bar{\mathcal{F}},\{\bar{\mathcal{F}}_t\}_{t \geq 0},\bar{\P})$, a cylindrical Wiener process $\bar{w}(t)$, and collection $\{\bar{\varphi}_\e\}_{\e \geq 0}$ in $L^2(\bar{\Omega};L^2(0,T;H))$ such that $\varphi_\e$ and $\bar{\varphi}_\e$ have the same distributions and 
 \begin{equation*}
 	\lim_{\e\to 0}\bar{\varphi}_\e = \bar{\varphi}_0, \ \P-\text{a.s.},
\end{equation*} 
with respect to the weak topology of $L^2(0,T;H)$. If we show that $z_{\e,\bar{\varphi}_\e}$ converges to $z_{\bar{\varphi}_0}$ in $C([0,T];L^4)$, $\P$-a.s., then condition (ii) will follow. 

To simplify our notation, we dispense with the bars. Now, for any $t \geq 0$, we have
	\begin{equation*}
		\begin{aligned}
		z_{\e,\varphi_\e}(t) - z_\varphi(t) &= \sqrt{\e}\int_0^t S(t-s)\sqrt{Q_\e}\,dw(s) + \int_0^t S(t-s)\Big[\sqrt{Q_\e}\varphi_\e(s) - \varphi(s) \Big]ds\\[10pt]	
		&=:J_1^\e(t) + J_2^\e(t).
		\end{aligned}
 \end{equation*}
	Thanks to Lemma \ref{lem_convolution_to0}, we have
	\[\lim_{\e\to 0} \mathbb{E} \Vert J^\e_1\Vert^4_{C([0,T];L^4)}=0.\]
	 To handle the control terms, we observe that $\sqrt{Q_\e}\,\varphi_\e$ converges to $\varphi$ weakly in $L^2(0,T;H)$. Indeed, for any $h \in L^2(0,T;H)$, it follows that
	\begin{align*}
		\Big|\langle \sqrt{Q_\e}\varphi_\e - \varphi, h \rangle_{L^2(0,T;H)} \Big| & = \Big| \langle \varphi_\e, (\sqrt{Q_\e}-I) h \rangle_{L^2(0,T;H)} + \langle \varphi_\e - \varphi,h \rangle_{L^2(0,T;H)} \Big|
		\\[10pt]
		& \leq \sqrt{2M}\norm{(\sqrt{Q_\e}-I) h}_{L^2(0,T;H)} + \Big| \langle \varphi_\e - \varphi,h \rangle_{L^2(0,T;H)} \Big|,
	\end{align*}
	which converges to $0$ $\P-$a.s., as $\e \to 0$, since $Q_\e$ converges to $\mathbbm{1}$ pointwise in $H$ and $\varphi_\e$ converges to $\varphi$ weakly. Moreover, we already showed in Step 1 that the solution map $\Gamma:L^2(0,T;H) \to C([0,T];L^4)$ given by
	\begin{equation*}
		\Gamma(\varphi)(t) = \int_0^t S(t -s) \varphi(s)ds, 
	\end{equation*} 
	is a compact operator. Since compact operators map weakly convergent sequences to strongly convergent sequences, it follows 
	\[\lim_{\e \to 0}\ \Vert J_2\Vert_{C([0,T];L^4)}=0,\ \ \ \ \mathbb{P}-\text{a.s.}\] 
	\end{proof}

\subsection{Uniform LDP for the Navier-Stokes process}
To obtain a uniform large deviations principle for the solutions to equation $\eqref{eq_NS_good}$, we will apply the contraction principle to the large deviations principle for the solutions to equation $\eqref{eq_OU_process}$. 

For every $x \in\,H$, let  $\F_x: L^4(0,T;L^4) \to C([0,T];H)$ be the family of mappings that associate to any $z \in L^4(0,T;L^4)$  the solution to the equation
\begin{equation*}
	\begin{cases}
		du(t) + Au(t)dt + B\Big(u(t)+z(t)\Big)dt = 0,
		\\[10pt] u(0) = x \in H.
	\end{cases}
\end{equation*}
In particular, we see that $I+\mathcal{F}_x$ maps a trajectory of $z_\e$ to a trajectory of $u_\e^x$. Throughout the remainder, we will use the shorthand 
\begin{equation*}
	B_Y(y,r):= \{h \in Y:\norm{h-y}_Y < r\}, \ \ \ \ \  y \in Y, \ \ \  r > 0,
\end{equation*}
for the open ball in the Banach space $Y$, centered at $y$ and of radius $r$. When $y=0$ we will just write $B_Y(r)$.

The proof of the following result is from \cite{bl}. Here we give a brief sketch of it to emphasize the right dependence on the initial conditions.

\begin{lemma}\cite{bl}
\label{lem_NS_cont}
	For every fixed $T>0$ the mappings $\F_x:L^4(0,T;H)\to C([0,T];L^4)$ are locally Lipschitz continuous, uniformly over $x$ in bounded sets of $H$. That is, for any $r > 0$ and $ R> 0$, there exists a constant $L_{r,R} > 0$ such that
	\begin{equation}\label{sn22}
		\sup_{x \in B_H(r)}\,\norm{\F_x(f) - \F_x(g)}_{C([0,T];H)} \leq L_{r,R} \norm{f-g}_{L^4(0,T;L^4)},\ \ \ \ \ f,g \in B_{L^4(0,T;L^4)}(R).
	\end{equation}
\end{lemma}
\begin{proof}
	For every $ f,g \in B_{L^4(0,T;L^4)}(R)$, we define $u:= \F_x(f)-\F_x(g)$.
	By proceeding as in \cite{bl}, we have
\begin{align*}
		\norm{u(t)}_H^2 & + \int_0^t \norm{u(s)}_V^2 ds \leq c \norm{g-f}_{L^4(0,t;L^4)}^2 \Big[ \norm{\F_x(f)}_{L^\infty(0,T;H)} \norm{\F_x(f)}_{L^2(0,T;V)} 
		\\[10pt] & +  \norm{\F_x(g)}_{L^\infty(0,T;H)} \norm{\F_x(g)}_{L^2(0,T;V)}   + \norm{f}_{L^4(0,T;L^4)}^2 + \norm{g}_{L^4(0,T;L^4)}^2 \Big]
		\\[10pt] & + c\int_0^t \Big[\norm{ \F_x(g)(s)}_V^2 + \norm{g(s)}_{L^4}^4\Big] \norm{u(s)}_H^2 ds.
	\end{align*}
	Now, for an arbitrary $f \in\,L^4(0,T;L^4)$
	\begin{equation}
	\label{sn20}
	\begin{aligned}
		&\frac{1}{2}\norm{\F_x(f)(t)}_H^2 + \int_0^t \norm{\F_x(f)(t)}_V^2ds 
		\\[10pt]  &\leq \frac{1}{2}\norm{x}_H^2 + \int_0^t \Big[ |b(\F_x(f)(s),f(s),\F_x(f)(s))| + |b(f(s),f(s),\F_x(f)(s))| \Big] ds
		\\[10pt]  &\leq  
		 \frac{1}{2}\norm{x}_H^2 + \int_0^t \Big[ \norm{\F_x(f)(s)}_H^{1/2}\norm{\F_x(f)(s)}_V^{3/2}\norm{f(s)}_{L^4} + \norm{f(s)}_{L^4}^2 \norm{\F_x(f)(s)}_V \Big] ds
		\\[10pt]  &\leq \frac{1}{2}\norm{x}_H^2 + \frac{1}{2} \int_0^t \norm{\F_x(f)(s)}_V^2ds +c\int_0^t \norm{f(s)}_{L^4}^4 \norm{\F_x(f)(s)}_H^2 ds + c\norm{f}_{L^4(0,t;L^4)}^4,
	\end{aligned}	
	\end{equation}
which implies that
	\begin{equation*}
		\norm{\F_x(f)(t)}_H^2 + \int_0^t \norm{\F_x(f)(s)}_V^2 ds \leq (\norm{x}_H^2 + \norm{f}_{L^4(0,t;L^4)}^4) \exp \Big( \norm{f}_{L^4(0,t;L^4)}^4 \Big).
	\end{equation*}
	This implies that if $f, g \in\,B_{L^4(0,T;L^4)}(R)$, and $x \in\,B_H(R)$, there exists $L_{r, R}>0$ such that
	\begin{equation}\label{eq_L4_cont}
		\begin{array}{l}
		\ds{	\norm{u(t)}_H^2  + \int_0^t \norm{u(s)}_V^2 ds}\\[10pt]
		\ds{\leq L_{r, R}  \norm{g-f}_{L^4(0,t;L^4)}^2 \exp \Big( c\int_0^t \Big[\norm{ \F_x(g)(s)}_V^2 + \norm{g(s)}_{L^4}^4\Big]ds\Big).}
		\end{array}
			\end{equation}
	By using again \eqref{sn20} to estimate $\F_x(g)$, we obtain \eqref{sn22}
\end{proof}

In the proof of the main result, Theorem \ref{th_main}, two different non-equivalent formulations of the uniform large deviations principle will be required. A thorough comparative analysis of the different formulations of uniform LDPs is given in the paper \cite{s1}. We state the definitions for the two forms needed in this paper, using the same notations and conventions as in \cite{s1}. This first definition can also be found in \cite{fw}.
\begin{definition}
	Let $E$ be a Banach space and let $D$ be some non-empty set. Suppose that for each $x \in D$,  $\{\mu_{\e}^x\}_{\e > 0}$ is a family of probability measures on $E$ and $I^x:E \to [0,+\infty]$ is a good rate function. The family $\{\mu_\e^x\}_{\e > 0}$ is said to satisfy a Freidlin-Wentzell uniform large deviations principle in $E$ with rate functions $I^x$, uniformly with respect to  $x \in\,D$, if the following statement holds.
	\begin{enumerate}[(i)]
		\item For any $s\geq 0$, $\delta > 0$ and $\gamma >0$, there exists $\e_0 > 0$ such that
		\begin{equation*}
			\inf_{x \in\,D}\left(\mu_\e^x( B_E(\varphi,\delta)) - \exp \Big( -\frac{I^x(\varphi) + \gamma}{\e}\Big)\right)\geq 0,\ \ \ \ \ \e \leq \e_0,
		\end{equation*}
		 for any  $\varphi \in \Phi^x(s)$, where $\Phi^x(s) := \{h \in E:I^x(h) \leq s\}$.
		\item For any $s_0 \geq 0$, $\delta > 0$ and $\gamma > 0$, there exists $\e_0 > 0$ such that
		\begin{equation*}
			\sup_{x \in\,D}\mu_\e^x(B^c_E(\Phi^x(s),\delta)) \leq \exp \Big( - \frac{s-\gamma}{\e}\Big),\ \ \ \e \leq \e_0,
		\end{equation*}
		for any $s \leq s_0$, where 
		\[B^c_E(\Phi^x(s),\delta) = \{h \in E: \mathrm{dist}_E(h,\Phi^x(s)) \geq \delta\}.\]
	\end{enumerate}
\end{definition}

We will use the following version of the contraction principle.
\begin{theorem}[Uniform Contraction Principle]\label{th_contraction}
	Assume the family of measures $\{\nu_\e\}_{\e > 0}$ satisfies a large deviations principle on Banach space $F$ with good rate function $J:F \to [0,+\infty]$. Suppose that $\{G^x\}_{x \in D}$ is a family of continuous mappings from $F$ to a Banach space $E$. Moreover, assume that $G^x$ are Lipschitz continuous, uniformly over $x \in D$, i.e. 
\begin{equation*}
	\sup_{x \in D} \sup_{\varphi_1 \neq \varphi_2} \frac{ \norm{ G^x(\varphi_1)-G^x(\varphi_2)}_E}{\norm{\varphi_V-\varphi_2}_F} =: L < \infty.
\end{equation*}
Then the family of push-forward measures $\{\mu_\e^x\}_{\e > 0}$ defined by $\mu_\e^x(\cdot):= \nu_\e((G^x)^{-1}(\cdot))$ satisfy a Freidlin-Wentzell uniform large deviations principle in $E$ with rate functions $I^x$ uniformly with respect to  $x \in\,D$, where $I^x$ is given by
\begin{equation*}
	I^x(\varphi) := \inf \{ J(\psi): \psi \in F, \ \varphi = G^x(\psi)\}.
\end{equation*}
\end{theorem}
\begin{proof}
	{\em Lower Bound.} Fix $s \geq 0$, $\delta > 0$ and $ \gamma > 0$. For each $x \in D$, let $\varphi^x \in E$ be such that $I^x(\varphi^x) \leq s$. Therefore, for each $x \in D$ there exists $\psi^x \in F $ such that $\varphi^x = G^x(\psi^x)$ and $J(\psi^x) \leq I^x(\varphi^x) + \gamma/2$. Since $J(\psi^x) \leq s + \gamma/2$, for each $x \in D$, we have
	\begin{align*}
		\mu_\e^x(B_E(\varphi^x,\delta)) & = \nu_\e( \{f \in F:\norm{G^x(f) -\varphi^x}_E < \delta\}) \geq \nu_\e\Big( \Big\{f \in F: \norm{f -\psi^x}_F < \frac{\delta}{L} \Big\}\Big)
		\\[10pt] & \geq \exp \Big(-\frac{ J(\psi^x) + \gamma/2}{\e} \Big) \geq \exp \Big(-\frac{ I^x(\varphi^x) + \gamma}{\e} \Big),
	\end{align*}	
	for any $\e \leq \e_0$ with $\e_0 > 0$ only depending on $s$, $\gamma$, $\delta$ and $L$. 
		
	 {\em Upper Bound.} Fix $s_0 \geq 0$, $\delta > 0$ and $\gamma > 0$ and observe that
	\begin{align*}
		\mu_\e^x(B^c_E(\Phi^x(s),\delta)) = \nu_\e \Big(\Big\{f \in F: \inf_{\varphi \in E\,:\,I^x(\varphi)\leq s} \norm{G^x(f) - \varphi}_E \geq \delta \Big\} \Big).
	\end{align*}
	Note that for a given $f \in F$, if there exists $\psi \in F$  such that $J(\psi) \leq s$ and $\norm{f-\psi}_E < \frac{\delta}{L}$, then $I^x(G^x(\psi)) \leq J(\psi) \leq s$ and $\norm{G^x(f) - G^x(\psi)} < \delta$ for any $x \in D$. Hence, there exists some $\e_0>0$ such that
	\begin{align*}
		\mu_\e^x(B^c_E(\Phi^x(s),\delta)) \leq \nu_\e \Big(\Big\{ f \in F: \inf_{\psi \in F\,:\,J(\psi)\leq s} \norm{f - \psi}_F \geq \frac{\delta}{L} \Big\} \Big) \leq \exp \Big( - \frac{s - \gamma}{\e}\Big),
	\end{align*}
	for any $\e \leq \e_0$.
\end{proof}
\color{black}
To define the Navier-Stokes rate function, we first define the Hamiltonian
\begin{equation*}
	\mathcal{H}(u):= u' + Au + B(u),\ \ \ \ \ u \in D(\mathcal{H}):=W^{1,2}(0,T;V^{-1}) \cap L^2(0,T;V)).
\end{equation*}
For $u \in\,D(\mathcal{H})$ in this space, the nonlinearity $B(u)$ is a well-defined element of $L^2(0,T;V^{-1})$. For any $x \in H$ and $u \in\,C([0,T];H)$, we define 
\begin{equation*}
		I^x(u) = 
		\begin{cases}
			\ds \frac{1}{2} \int_0^T\norm{\mathcal{H}(u)(t)}_H^2dt & \text{ if }  \mathcal{H}(u) \in L^2(0,T;H),\ \text{and}\ u(0) = x,
			\\[10pt] +\infty & \text{ otherwise}.
		\end{cases}		
	\end{equation*}
\begin{theorem}\label{th_FWLDP}
Assume that 
	\begin{equation*}
		\lim_{\e \to 0} \delta(\e) =0,\ \ \ \ \lim_{\e \to 0} \e \log \frac{1}{\delta(\e)} = 0.
	\end{equation*} 
	If $u^x_\e$ is the solution to equation \eqref{eq_NS_good}, then for any $R > 0$, the family $\{\mathcal{L}(u_\e^x) \}_{\e > 0}$  satisfies a Freidlin-Wentzell uniform large deviations principle in $C([0,T];H)$ with rate functions $I^x$, uniformly with respect to  $x \in\,B_H(R)$.
\end{theorem}
\begin{proof}
	First of all, notice that $u^x_\e=(I+\F_x)(z_\e)$. Lemma $\ref{lem_NS_cont}$ implies that the mapping 
	\[I+\F_x:C([0,T];L^4) \to C([0,T];H),\] is locally Lipschitz, uniformly on bounded sets. Therefore, thanks to the contraction principle, Theorem $\ref{th_contraction}$, and Theorem \ref{th_OULDP}, the family $\{(I+\F_x)(z_\e)\} = \{u_\e^x\}$ satisfies a Freidlin-Wentzell uniform large deviations principle with rate function
	\begin{align*}
		I_T^x(u) &	= \inf\Big\{J_T(z):u = z + \F_x(z), \ z \in W^{1,2}(0,T;H) \cap L^2(0,T;D(A)) \Big\}.
	\end{align*}
	If $ u \in D(\mathcal{H})$ and $u(0) = x$, then $\mathcal{H}(u) \in L^2(0,T;V^{-1})$ and $u$ is a weak solution  to
	\begin{equation}\label{eq_NS_forced_det}
	\begin{cases}
		du(t) + [Au(t) +B(u(t))]dt = \mathcal{H}(u)(t)dt,
		\\[10pt] u(0) = x.			
	\end{cases}
\end{equation}	 
Note that $u \in D(\mathcal{H})$ implies that equation $\eqref{eq_NS_forced_determ}$ with forcing $\varphi = \mathcal{H}(u)$ has a unique weak solution $z_\varphi \in X$. In particular this also implies that $\mathcal{F}_x(z_\varphi) \in D(\mathcal{H})$ and $u = z_\varphi + \mathcal{F}_x(z_\varphi)$. This decomposition is unique. Indeed, if $u = z + \mathcal{F}_x(z)$ for some other  $z \in D(\mathcal{H})$, then $u-\mathcal{F}_x(z)$ would again be a weak solution to equation $\eqref{eq_NS_forced_determ}$ with forcing $\varphi = \mathcal{H}(u)$ so that $z = z_\varphi$. This implies that 
\begin{equation*}
	J_T(z_\varphi) = \frac{1}{2} \int_0^T \norm{\mathcal{H}(u)(t)}_H^2dt, 
\end{equation*}
whenever $\mathcal{H}(u) \in L^2(0,T;H)$.
\end{proof}
\begin{remark}
{\em 	Notice that both the proof of Theorem \ref{th_OULDP} and the proof of Theorem \ref{th_FWLDP} do not require periodic boundary conditions.}
\end{remark}

The second definition for the uniform large deviations principle is given by the following. It can be found in \cite{dz3}.
\begin{definition}
	Let $E$ be a Banach space and let $D$ be some non-empty set. Suppose that for each $x \in D$,  $\{\mu_{\e}^x\}_{\e > 0}$ is a family of probability measures on $E$ and $I^x:E \to [0,+\infty]$ is a good rate function. The family $\{\mu_\e^x\}_{\e > 0}$ is said to satisfy a Dembo-Zeitouni uniform large deviations principle in $E$ with rate functions $I^x$ uniformly with respect to  $x \in\,D$ if the following hold.
	\begin{enumerate}[(i)]
		\item For any $\gamma >0$ and open set $G \subset E$, there exists $\e_0 > 0$ such that
		\begin{equation*}
			\inf_{x \in\,D}\,\mu_\e^x(G) \geq \exp \left( -\frac{1}\e\left[\, \sup_{y \in D}\, \inf_{u \in G} I^y(u) + \gamma\right]\right),\ \ \ \ \e\leq \e_0.
		\end{equation*}
		
		\item For any $\gamma > 0$ and closet set $F \subset E$, there exists $\e_0 > 0$ such that
		\begin{equation*}
			\sup_{x \in\,D}\,\mu_\e^x(F) \leq \exp \left( - \frac{1}{\e}\left[\inf_{y \in D}\, \inf_{u \in G} I^y(u)-\gamma\right]\right),\ \ \ \ \e\leq \e_0.
		\end{equation*}
	\end{enumerate}
\end{definition}
\begin{corollary}\label{cor_DZLDP}
Assume that 
	\begin{equation*}
		\lim_{\e \to 0} \delta(\e) = 0,\ \ \ \ \lim_{\e \to 0} \e \log \frac{1}{\delta(\e)} = 0.
	\end{equation*} 
	Let $K \subset H$ be a compact set. Then the family $\{\mathcal{L}(u_\e^x) \}_{\e > 0}$ of solutions to equation $\eqref{eq_NS_good}$ satisfies a Dembo-Zeitouni uniform large deviations principle in $C([0,T];H)$ with rate functions $\{I^x\}_{x \in K}$ uniformly with respect to  $x \in\,K$.
\end{corollary}
\begin{proof}
	In view of Theorem 2.7 of \cite{s1}, to prove equivalence of the two uniform large deviation principles over a compact subset of $H$, it suffices to show that for every fixed $s\geq 0$ the mapping
	\[x \in\,H\mapsto  \Phi^x(s):= \{u \in C([0,T];H):I^x(u) \leq s \},\] is continuous with respect to the Hausdorff metric. That is, we must show that for any  $\{x_n\}_{n =1 }^\infty \subset H$ such that $x_n \to x \in H$, 
	\begin{equation*}
		\lim_{n \to \infty}  \max \left( \sup_{u \in \Phi^{x_n}(s)} \mathrm{dist}_{C([0,T];H)}(u, \Phi^x(s)), \sup_{u \in \Phi^{x}(s)} \mathrm{dist}_{C([0,T];H)}(u, \Phi^{x_n}(s))\right)   = 0.
	\end{equation*}	 
	This is immediately implied by the continuity of the Navier-Stokes equations with respect to initial conditions. Indeed, suppose that $u^x_\varphi$ is a  solution to the equation, 
	\begin{equation}\label{eq_NS_forced}
		\begin{cases}
			du(t) + (Au(t)+B(u(t)))dt = \varphi(t)dt,
			\\[10pt] u(0) = x,
		\end{cases}
	\end{equation} 
	with  driving force $\varphi \in L^2(0,T;H)$. Then by standard energy estimates (see for instance, \cite{ks}), we have
	\begin{align*}
		\sup_{0 \leq t \leq T}\norm{u_\varphi^x(t) - u_\varphi^y(t)}_H^2 & + \int_0^T \norm{u_\varphi^x(t) - u_\varphi^y(t)}_V^2 dt  
		\\[10pt] & \leq  \norm{x-y}_H^2\exp\Big(c \int_0^T\norm{u^y_\varphi(t)}_V^2 dt \Big)
		\\[10pt]  & \leq \norm{x-y}_H^2 \exp \Big( c \big[\norm{y}_H^2 + \norm{\varphi}_{L^2(0,T;H)}^2 \big] \Big).
	\end{align*}
	Now, if $u \in \varphi^x(s)$, then $\varphi_u:=\mathcal{H}(u) \in L^2(0,T;H)$, $\frac{1}{2}\norm{\varphi_u}_{L^2(0,T;H)}^2 \leq s$ and $u$ solves equation $\eqref{eq_NS_forced_det}$. But then, the weak solution $v \in W^{1,2}(0,T;V^{-1}) \cap L^2(0,T;V)$ to
	\begin{equation*}
		\begin{cases}
			dv(t) + [Av(t) + B(v(t))]dt = \varphi_u(t)dt,
			\\[10pt]  v(0) = y,
		\end{cases}
	\end{equation*}
	belongs to $\Phi^y(s)$. Therefore, 
	\begin{equation*}
		 \mathrm{dist}_{C([0,T];H)}(u,\Phi^y(s) )\leq \norm{u-v}_{C([0,T];H)} \leq c_{s}(\norm{y}_H)\,\norm{x-y}_H,
	\end{equation*}
	for some continuous increasing function $c_s:[0,+\infty)\to [0,+\infty)$.
	Since this is true for arbitrary $u \in \Phi^x(s)$, it follows that
	\begin{equation*}
		\sup_{u \in \Phi^x(s)} \mathrm{dist}_{C([0,T];H)}(u,\Phi^y(s) ) \leq c_{s}(\norm{y}_H)\,\norm{x-y}_H,
	\end{equation*}
	which implies the result, since $\sup_{n \in\,\mathbb{N}}\norm{x_n}_H<\infty$.
\end{proof}

\section{Proof of Theorem \ref{th_main}}

We start this section with the description of the quasi-potential associated with  equation $\eqref{eq_NS_good}$. To simplify notation, for any $ T > 0$ we will denote
\begin{equation*}
	I_T(u):= \frac{1}{2} \int_0^T \norm{\mathcal{H}(u)(t)}_H^2 dt,
\end{equation*}
whenever $\mathcal{H}(u) \in L^2(0,T;H)$. In addition, we set 
\[I_T^y(u):=\begin{cases}
	I_T(u), & \text{if}\ u(0) = y,\\
	+\infty,  &  \text{otherwise.}
\end{cases}\]
The quasi-potential, $U:H \to [0,+\infty]$ is  defined as
\begin{equation}\label{eq_quasi}
	U(x) := \inf \{ I_{T}(u): T > 0, u \in C([0,T];H), u(0) = 0, u(T) = x\}.
\end{equation}
For any $x \in H$, $U(x)$ gives the minimum action of all paths that start at $0$ and end at $x$. Since $0$ is an asymptotically attracting equilibria for the Navier-Stokes equations, $U(x)$ will govern the long-time dynamics and asymptotic behavior of the invariant measures.

In the particular case of the Navier-Stokes equations on the torus, the orthogonality of $B(u)$ and $Au$ can be taken advantage of to provide an explicit formula for the quasipotential. In fact, as proven in  \cite[Theorem 7.1]{bcf} we have that	for any $x \in H$ 
	\begin{equation}
	\label{sn1}	
	U(x) = \begin{cases}
		\norm{x}_V^2, & x \in V,
		\\[10pt] + \infty, & x \in H \setminus V.
		\end{cases}
	\end{equation}

\medskip 

Now, we proceed with the proof of 
Theorem $\ref{th_main}$.  Some of the steps of the proof are analogous to those used in   \cite[Theorem 4.5]{bc}, where a large deviation principle for the invariant measures of the $2D$ stochastic Navier-stokes equation is studied, under the assumption that the covariance of the noise does not depend on $\e$.
In those steps our arguments will be less detailed and we refer the reader to \cite{bc}. On the other hand, our arguments will be fully detailed in those steps of the proof that deviate from \cite{bc}, and require new arguments and techniques.

\subsection{Lower bound}
\begin{proposition}
	Under the assumptions of Theorem $\ref{th_main}$, the family of invariant measures $\{\nu_\e\}_{\e > 0}$ of equation $\eqref{eq_NS_good}$ satisfies the large deviations principle lower bound in $H$ with rate function $U(x)$. That is, for any $x \in H$, $\delta > 0$ and $\gamma > 0$, there exists $\e_0>0$ such that
	\begin{equation*}
		\nu_\e(B_H(x,\delta)) \geq \exp \Big(-\frac{U(x) + \gamma}{\e} \Big),\ \ \ \ \ \e \leq \e_0.
	\end{equation*}
	\end{proposition}
\begin{proof}
	Fix  $x \in H$, and any $\delta > 0$, $\gamma > 0$ and $T > 0$. We assume that $U(x) < \infty$ or else there is nothing to prove. Suppose that $\{v^y\}_{y \in H} \subset C([0,T];H)$ is a family of paths satisfying 
	\[ \sup_{y \in H} \norm{v^y(T) - x}_H < \delta/2.\]
	Thanks to the invariance of $\nu_\e$,  we have
	\begin{align*}
		\nu_\e(B_H(x,\delta)) & = \int_H \P( \norm{u^y_\e(T)-x}_H < \delta) \ d\nu_\e(y) 
		\\[10pt] & \geq \int_H \P(\norm{u_\e^y - v^y}_{C([0,T];H)} < \delta/2) \ d\nu_\e(y)
		\\[10pt] & \geq \int_{B_H(0,R)} \P(\norm{u_\e^y - v^y}_{C([0,T];H)} < \delta/2) \ d\nu_\e(y)
		\\[10pt] 	& \geq \nu_\e(B_H(0,R)) \inf_{y \in B_H(0,R)} \P(\norm{u_\e^y - v^y}_{C([0,T];H)} < \delta/2).
	\end{align*}
	Since the invariant measures are becoming concentrated around $0$, as $\e \downarrow 0$, we have
	\begin{equation*}
		\lim_{\e \to 0} \nu_\e(B_H(0,R)) = 1,
	\end{equation*}
	 for any $R > 0$. Thus, we can  pick $\e_1(R)>0$ small enough that \[\nu_\e(B_H(0,R)) \geq \frac{1}{2},\ \ \ \ \e \leq \e_1(R).\] Thanks to Theorem \ref{th_FWLDP}, a Freidlin-Wentzell uniform large deviations principle holds. Then,   for every $s_0 > 0$ there exists $\e_2(R)>0$  such that for any $v^y \in\,C([0,T];H)$ with $I^y_{T}(v^y) \leq s_0$,	\begin{equation*}
		\inf_{y \in B_H(0,R)} \P\left(\norm{u_\e^y - v^y}_{C([0,T];H)} < \delta/2\right) \geq \inf_{y \in B_H(0,R)} \,\exp \left( -\frac{1}\e \left[I^y_{T}(\varphi^y) + \gamma/2\right]\right),
	\end{equation*}
	for every $\e\leq \e_2(R)$.
	Therefore, to complete the proof, it remains to find a $T$ large enough that for each $y\in B_H(0,R)$, there exists a path $v^y \in C([0,T];H)$ with $v^y(0) = y$ that satisfies
	\begin{enumerate}[(a)]
		\item $I_T(v^y) \leq U(x) + \gamma/2$,
		\item $\norm{v^y(T) - x}_H < \delta/2$.
	\end{enumerate}
	The paths we choose are the solutions $u^y_\varphi$ to the controlled Navier Stokes equations, equation $\eqref{eq_NS_forced}$, with initial condition $y \in H$ and control $\varphi \in L^2(0,T;H)$, defined by
	\begin{align*}
		\varphi(t) = 
		\begin{cases}
			0 & \text{ if } 0 \leq t \leq T_1,
			\\[10pt]  \bar{\varphi}(t-T_1) & \text{ if } T_1 \leq t \leq T_1+ T_2,
		\end{cases}
	\end{align*}
	with $T_1$ and $T_2$ to be chosen.
	Here,  $\bar{\varphi} \in C([0,T_2];H)$ is a path such that $u^0_{\bar{\varphi}}(0) = 0$ and $u^0_{\bar{\varphi}}(T_2)=x$  with $I_{T_2}(u^0_{\bar{\varphi}})\leq  U(x) + \gamma/2$. Such a $T_2$ and $\bar{\varphi}$ exist by the definition of the quasipotential $U$. Meanwhile, $T_1 = T_1(\lambda)$ is taken large enough that the solutions $\{u_0^y\}_{ y \in B_H(0,R)}$ to the unforced Navier-Stokes equations satisfy
	\begin{equation*}
		\sup_{ y \in B_H(0,R)} \norm{u_0^y(T_1)}_H < \lambda,
	\end{equation*}
	for some small $\lambda$. Clearly point (a) is satisfied since the path contributes nothing to the action integral on the interval $[0,T_1]$. Point (b) follows by noting that the controlled Navier Stokes equations are continuous with respect to initial conditions. Indeed, since $u^y_0(T_1 ) \in B_H(0,\lambda)$, we have by a standard estimate (for example see Proposition 2.1.25 of \cite{ks}) that
	\begin{align*}
		\norm{u^y_\varphi(T_1+T_2) - x}_H & \leq  \sup_{z \in B_H(0,\lambda)} \norm{u^z_{\bar{\varphi}}(T_2) - u^0_{\bar{\varphi}}(T_2)}_H 
		\\[10pt]  & \leq \sup_{z \in B_H(0,\lambda)} \norm{z}_H \exp \Big({c\norm{z}_H^2 + c \norm{\bar{\varphi}}_{L^2(0,T_2;H)}^2}\Big).
	\end{align*}
	This implies point (b) if $\lambda$ is taken small enough. We conclude the proof upon taking $\e_0:= \min (\e_1,\e_2)$.
	
	\end{proof}

\subsection{Upper bound}
\begin{proposition}\label{prop_upper}
Under the assumptions of Theorem $\ref{th_main}$, the family of invariant measures $\{\nu_\e\}_{\e > 0}$ of equation $\eqref{eq_NS_good}$ satisfies the large deviations principle upper bound in $H$ with rate function $U(x)$. That is, for any $s \geq 0$, $\delta > 0$ and $\gamma > 0$, there exists $\e_0 > 0$ such that
	\begin{equation*}
		\nu_\e\left(\{h \in H:\mathrm{dist}_H(h,\Phi(s)) > \delta\}\right) \leq \exp \left( -\frac{s-\gamma}{\e} \right),\ \ \ \ \ \e \leq \e_0.
	\end{equation*}
	where 
	\[\Phi(s):= \{y \in H:U(y)\leq s \}.\]
\end{proposition}

 The proof requires the following three lemmas. 
\begin{lemma}[{\em Exponential Estimate}]\label{lem_exp_est}
	Assume that $Q$ has the form given in $\eqref{eq_Q_cov}$ for some $\beta > 2$. Moreover, suppose that 
	\begin{equation*}
		\lim_{\e \to 0} \delta(\e) = 0,\ \ \ \ \ \lim_{\e \to 0} \e \delta(\e)^{-2/\beta} = 0.
	\end{equation*}
	Then for any $s > 0$ there exist $\e_s >0$ and $ R_s > 0$ such that
	\begin{equation*}
		\nu_\e(B_V(0,R_s))\geq 1 - \exp\left(-\frac{s}{\e} \right),\ \ \ \ \ \e \leq \e_s.
	\end{equation*}
	\end{lemma}
\begin{proof}
	Fix  $R > 0$, $\e > 0$ and $\gamma > 0$ and let $u_\e^0$ be the solution of equation $\eqref{eq_NS_good}$. Thanks to the ergodicity of $\nu_\e$, we have
	\begin{align}\label{eq_exp_est_erg}
		\nu_\e(B_V^c(0,R))& = \lim_{T \to \infty} \frac{1}{T} \int_0^T \P(u^0_\e(s) \in B_V^c(0,R))\, ds \nonumber
		\\[10pt]  & \leq \exp \Big(-\frac{R^2 }{2\e} \Big) \frac{1}{T}\limsup_{T \to \infty} \int_0^T \E \exp \left( \frac{ \norm{u_\e^0(s)}_V^2}{2\e} \right )ds.
	\end{align}
	To estimate the expectation of the exponential, we apply the Ito formula to the functional $F:\R \times V \to \R$ defined by
	\begin{equation*}
		F(t,v) = \exp\left(t+\frac{ \norm{v}_V^2}{2\e}\right),
	\end{equation*}
	whose derivatives are given by
	\[D_t F(t,u) = F(t,u),\]
	and
	\[D_u F(t,u) = \frac{1}{\e} F(t,u) u,\ \ \ \ 
	D^2_u F(t,u) = \frac{1}{\e^2} F(t,u)u \otimes u + \frac{1}{\e} F(t,u) I.\]
	Formal application of the Ito formula to the solution $u_\e^x$ to equation $\eqref{eq_NS_good}$ implies that
	\begin{align*}
		\E F(t,u_\e^x(t)) & = F(0,x) +\E \int_0^t \Big[D_t F(s,u_\e^x(s)) + \langle D_u F(s,u_\e^x(s)), -Au_\e^x(s) - B(u_\e^x(s))\rangle_V
		\\[10pt] & + \frac{\e}{2} \sum_{k \in \Z20}^\infty \langle D^2_u F(s,u_\e^x(s))Q_\e e_k,e_k \rangle_V \Big] ds
		\\[10pt] & = F(0,x) + \E \int_0^t F(s,u_\e^x(s)) \Big[ 1-\frac{1}{\e} \norm{u^x_\e(s)}_{V^2}^2 
		\\[10pt] & +  \frac{\e}{2} \sum_{k\in \Z20} \frac{1}{\e^2} \left( \left|\langle u_\e^x(s),Q_\e e_k \rangle_V\right|^2  + \frac{1}{\e} \langle Q_\e e_k,e_k\rangle_V \right) \Big]\,ds
		\\[10pt] & = F(0,x) + \E \int_0^t F(s,u_\e^x(s)) \Big[ 1-\frac{1}{\e} \norm{u^x_\e(s)}_{V^2}^2 
		\\[10pt] & +  \frac{1}{2} \sum_{k\in \Z20} \Big( \frac{1}{\e}\sigma_{\e,k}^2 |\langle u_\e^x(s), Ae_k \rangle_H|^2 + |k|^2 \sigma_{\e,k}^2 \Big) \Big]ds
		\\[10pt] & \leq F(0,x) + \E \int_0^T F(s,u_\e^x(s)) \Big[ 1-\frac{1}{2\e} \norm{u_\e^x(s)}_{V^2}^2 
 + \frac{1}{2} \sum_{k \in \Z20} |k|^2\sigma_{\e,k}^2 \Big]ds,
	\end{align*}
		where in the second line we used identity $\eqref{eq_prop_non_Au}$ to dispose of the nonlinearity and in the fourth line we used that $|\sigma_{\e,k}| \leq 1$, for any $k \in \Z20$ and $\e > 0$. 
		
		Now, since $\beta > 2$, we have 
	\begin{equation*}
			P_\e := \sum_{k \in \Z20} |k|^2\sigma_{\e,k}^2 = \sum_{k \in \Z20} \frac{|k|^2}{1+\delta(\e) |k|^{2\beta}} \leq c \int_1^\infty \frac{r}{1+\delta(\e)r^{\beta}}\,dr\leq c \ \delta(\e)^{-2/\beta}.
	\end{equation*}
	Therefore, thanks to the Poincar\'e inequality and the fact that $e^x(a-x) \leq \exp(a-1)$, for every $a>1$ and $x\geq 0$, it follows that
	\begin{align*}
		\E F(t,u_\e^x(t)) & \leq \exp \Big( \frac{\norm{x}_V^2}{2\e} \Big) + \E \int_0^t \exp(s) \exp\Big(\frac{\norm{u_\e^x(s)}_V^2}{2\e}\Big) \Big(1+\frac{1}{2}P_\e - \frac{\norm{u_\e^x(s)}_V^2}{2\e}\Big)ds
		\\[10pt]  & \leq \exp \left( \frac{\norm{x}_V^2}{2\e} \right)  + \int_0^t \exp(s) \exp\left( \frac{1}{2}P_\e\right) ds.
	\end{align*}
	Hence,
	\begin{equation*}
		\E \exp \Big(\frac{\norm{u_\e^x(t)}_V^2}{2\e} \Big)\leq \exp \Big( -t + \frac{\norm{x}_V^2}{2\e} \Big) + \exp\Big( \frac{1}{2}P_\e\Big) .
	\end{equation*}
	Finally, using equation $\eqref{eq_exp_est_erg}$, we see that
	\begin{align*}
		\nu_\e(B_V^c(0,R)) & \leq \exp \Big(-\frac{R^2}{\e} \Big) \limsup_{T\to \infty} \frac{1}{T} \int_0^T \Big[ e^{-t} + \exp\Big( \frac{1}{2}P_\e\Big)\Big]dt
		\\[10pt] & = \exp \Big(-\frac{R^2}{\e} + \frac{P_\e}{2} \Big) \leq \exp \Big( - \frac{R^2- C\e\, \delta_\e^{-2/\beta}}{\e} \Big),
	\end{align*}
	which completes the proof of the lemma, since $\e\,\delta(\e)^{-2/\beta}\to 0$, as $\e\downarrow 0$.
		
\end{proof}
\begin{lemma}\label{lem_smallLambda}
	For any $\delta > 0$ and $s > 0$, there exist $\lambda > 0$ and $T > 0$ such that for any $t \geq T$ and $z \in C([0,t];H)$,
	\begin{equation*}
		|z(0)|_H < \lambda,\ \ I_{T}(z) \leq s \implies \mathrm{dist}_H(z(t),\Phi(s)) < \delta,
	\end{equation*}
	where $\Phi(s):= \{x \in H:U(x) \leq s \}$.
\end{lemma}
\begin{lemma}\label{lem_Henergy} For any $s >0$, $\delta >0$ and $r > 0$, let $\lambda$ be as in Lemma \ref{lem_smallLambda}. Then there exists $N \in \mathbb{N}$ large enough that
	\begin{equation*}
		u \in H_{r,s,\delta}(N) \implies I_{T}(u) \geq s,
	\end{equation*}
	where the set $H_{r,s,\delta}(n)$ is defined for $N \in \mathbb{N}$ by
	\begin{equation*}
		H_{r,s,\delta}(N) := \left\{u \in C([0,N];H),\ \norm{u(0)}_H \leq r,\ \norm{u(j)}_H \geq \lambda,\ j = 1,...,N\right\}.
	\end{equation*}
\end{lemma}
\noindent The proofs of Lemma $\ref{lem_smallLambda}$ and $\ref{lem_Henergy}$ depend only on the properties of the deterministic Navier-Stokes equation and can be found in \cite{bc} (see Lemmas 7.2 and 7.3). 
\begin{proof}[Proof of Proposition \ref{prop_upper}]
	Fix any $s > 0$, $\delta > 0$ and $\gamma > 0$ and let $R_s$ be as in Lemma \ref{lem_exp_est}. Due to the invariance of $\nu_\e$, for any $t \geq 0$ we have
	\[\begin{array}{l}
		\ds \nu_\e\left(\left\{ h \in H: \mathrm{dist}_H(h,  \Phi(s)) \geq \delta\right\}\right)  = \int_H \P\left(\mathrm{dist}_H(u^y_\e(t),\Phi(s))\geq \delta\right) d\nu_\e(y)
		\\[16pt] 
		 \ds = \int_{B_V^c(0,R_s)}\P\left(\mathrm{dist}_H(u^y_\e(t),\Phi(s))\geq \delta\right) d\nu_\e(y) 
		\\[16pt] 
		 \ds + \int_{B_V(0,R_s)} \P\left(\mathrm{dist}_H(u^y_\e(t),\Phi(s))\geq \delta , u_\e^y \in H_{R_s,s,\delta}(N)\right)d\nu_\e(y)
		\\[16pt] 
		 \ds + \int_{B_V(0,R_s)} \P\left( \mathrm{dist}_H(u^y_\e(t),\Phi(s))\geq \delta , u_\e^y \notin H_{R_s,s,\delta}(N) \right)d\nu_\e(y)
		\\[16pt] 
		 \ds =: K_1+K_2+K_3.
	\end{array}\]
	Now, thanks to Lemma \ref{lem_exp_est} we know that
	\begin{equation*}
		K_1 \leq \nu_\e (B_V^c(0,R_s)) \leq \exp \Big(-\frac{s}{\e} \Big).
	\end{equation*}
	Next, let $N$ be as in Lemma \ref{lem_Henergy}. Since $H_{R_s,s,\delta}(N)$ is a closed set in $C([0,N];H)$ and $B_V(0,R_s)$ is a compact subset of $H$, the Dembo-Zeitouni uniform large deviation principle over compact sets, Corollary \ref{cor_DZLDP}, implies that there exists $\e_0>0$ such that
	\begin{align*}
		K_2 & \leq \sup_{y \in B_V(0,R_s)} \P(u_\e^y \in H_{R_s,s,\delta}(N)) 
		\\[10pt] & \leq \exp \Big( -\frac{1}{\e} \Big[ \inf_{z \in B_V(0,R_s)} \inf_{h \in H_{R_s,s,\delta}(N)} I_{T}^z(h) - \gamma \Big] \Big) ,
	\end{align*}
	for any $\e \leq \e_0$.	Hence, by Lemma \ref{lem_Henergy}, \[K_2 \leq \exp(-\frac{1}{\e}[s-\gamma]).\] To address $K_3$, we use the Markov property of $u_\e$ to stop the process at integer times. We then have
	\begin{align*}
		K_3 & = \int_{B_V(0,R_s)} \P \left(\bigcup_{j = 1}^{N} \left\{ |u^y_\e(j)|_H < \lambda\right\} \bigcap \left\{\mathrm{dist}_H(u^y_\e(t),\Phi(s))\geq \delta \right\}\right)d\nu_\e(y)
		\\ & \leq \sum_{j=1}^{N} \int_{B_V(0,R_s)} \P \left( \left\{|u^y_\e(j)|_H < \lambda \right\}\bigcap \left\{\mathrm{dist}_H(u^y_\e(t),\Phi(s))\geq \delta\right\}\right) d\nu_\e(y)
		\\ & \leq  \sum_{j=1}^N \sup_{y \in B_H(0,\lambda)} \P( \mathrm{dist}_H(u_\e^y(t-j),\Phi(s)) \geq \delta).
	\end{align*}
	In order to use the uniform LDP of Theorem \ref{th_FWLDP}, we must convert this event at time $t-j$ to an event in $C([0,t-j];H)$. To do so, we pick $t$ large enough that Lemma \ref{lem_smallLambda} applies for $\delta/2$. Then, if $y \in\,B_H(\lambda)$
	\begin{align*}
		 \mathrm{dist}_H(u_\e^y(t-j), & \Phi(s)) \geq \delta 
		 \\[10pt]  & \implies \inf \Big\{ \norm{u_\e^y -v}_{C([0,t-j];H)}:\norm{v(0)}_H < \lambda, I_{T}(v)\leq s \Big\} \geq \frac{\delta}{2}	
		 \\[10pt]  & \implies \mathrm{dist}_{C([0,t-j];H)}\Big(u_\e^y, \Psi^y(s) \Big) \geq \delta/2,
	\end{align*}
	where 
	\[\Psi^y(s):=\{v \in C([0,t-j];H):v(0) = y, I_{T}(v) \leq s \}.\] 
	Then, by Theorem \ref{th_FWLDP}, there exists $\e_{0,j}$ such that for any $\e \leq \e_{0,j}$,
	\begin{align*}
		\sup_{y \in B_H(0,\lambda)} & \P( \mathrm{dist}_H(u_\e^y(t-j),\Phi(s))  		 \geq \delta) 
		\\ & \leq  \sup_{y \in B_H(0,\lambda)} \P ( \mathrm{dist}_{C([0,t-j];H)}\Big(u_\e^y, \Psi^y(s) \Big) \geq \delta/2)
		\\ & \leq \exp \Big(- \frac{s-\gamma}{\e} \Big).
	\end{align*}
	Hence, for any $\e < \min(\e_0,\e_{0,1},...,\e_{0,N})$ it follows that
	\begin{equation*}
		K_3 \leq N\exp \Big(- \frac{s-\gamma}{\e} \Big),
	\end{equation*}
	which implies the result.
\end{proof}

\end{document}